\documentclass[10pt]{article}
\usepackage{amsmath,amssymb,enumerate} 

\usepackage{amsthm}

\usepackage{tikz}
\usetikzlibrary{arrows}
\usetikzlibrary{positioning}

\newtheorem{theorem}{Theorem}

\newtheorem{proposition}{Proposition}
\newtheorem{corollary}{Corollary}

\def\ZZ{\mathbb{Z}}
\def\ZZ1{\mathbf{Z}_{\geq 1}}

\def\RR{\mathbb{R}}

\def\bx{\mathbf{x}}

\def\by{\mathbf{y}}

\title{Balanced flows for transshipment problems
}

\author{
Vladimir Gurvich\thanks{
HSE University, Moscow Russia; e-mail:
vgurvich@hse.ru ; vladimir.gurvich@gmail.com}
}

\begin{document}

\maketitle

\abstract{
A transshipment problem $(G, d, \lambda)$  is modeled
by a directed graph  $G = (V, E)$
with weighted vertices (nodes) $d = (d_v \mid v \in V)$  and
weighted directed edges (arcs) $\lambda = (\lambda_e \mid e \in E)$
interpreted as follows:
$G$  is a communication or transportation network, for example, a pipeline;
each arc  $e \in E$  is a one-way communication line, road or pipe
of capacity  $\lambda_e$, while
every vertex  $v \in V$  is a node of production
$d_v > 0$, consumption $d_v < 0$, or just transition $d_v = 0$.
A non-negative flow  $x = (x_e \mid e  \in E)$  is called
{\em weakly feasible} if 
for each  $v \in V$  the algebraic sum of flows,
over all directed edges incident to $v$,  equals $d_v$,
or shorter, if  $A_G x = d$,   where  $A_G$  is
the vertex-edge incidence matrix of $G$.
A weakly feasible flow  $x$  is called {\em feasible}
if  $x_e \leq  \lambda_e$  for all  $e \in E$.
We consider weakly feasible but not necessarily feasible flows, that is,
inequalities  $x_e > \lambda_e$  are allowed.
However, such an excess is viewed as unwanted (dangerous) and so
we minimize the excess ratio vector
$r = (r_e = x_e  / \lambda_e  \mid  e \in E)$  lexicographically.
More precisely, first, we look for weakly feasible flows
minimizing the maximum of  $r_e$  over all $e \in E$;
among all such flows we look for those that
minimize the second largest coordinate of $r$, etc.
Clearly, in $|E|$  such steps obtain a unique lexmin vector  $r$.
We will show that the corresponding {\em balanced flow} is also unique and
represents the {\em lexmin solution} for problem  $(G, d, \lambda)$.
Using the Gale-Hoffman inequalities, we construct it in polynomial time,
provided   vectors  $d$  and  $\lambda$  are rational.
\newline
For symmetric digraphs the problem was solved
by Gurvich and Gvishiani in 1984.
Here we extend this result to arbitrary directed graphs.
Furthermore, we simplify the algorithm and proofs
applying the classic criterion of existence of
a feasible flow for $(G, d, \lambda)$
obtained by Gale and Hoffman in late 1950-s.
\newline
MSC classes:
90C25, 
94C15, 
94C99. 
}

\newpage

\section{Lexmin Linear Programming}
\label{s0}
For a positive integer $m$   define by  $[m]=\{1,...,m\}$
the set of positive integers up to $m$, by 
$\RR_{\geq 0}^m$  
the set of nonnegative  $m$-vectors,  and by $\RR_{> 0}^m$  
the $m$-vectors whose all coordinates are strictly positive.

Let $K \subseteq \RR_{\geq 0}^m$ be a convex polyhedral region.
We can view a solution $\bx\in K$ as a possible operational state of a physical system,
where the components $\bx=(x_1,...,x_m)$ are the corresponding levels of consumption
of certain resources, $i\in [m]$. We shall also consider  positive thresholds $\lambda_i>0$, $i\in [m]$,
describing the desirable levels of resource consumptions, and
define $\lambda=(\lambda_1,...,\lambda_m)\in \RR_{> 0}^m$. To a feasible vector $\bx\in K$ we associate
\[
r(\bx) ~=~ \left(\frac{x_i}{\lambda_i}\mid i\in [m]\right)
\]
in which
the ratios $r_i(\bx)=\frac{x_i}{\lambda_i}$, $i\in [m]$, 
describe the utilization levels of resources at the operating state $\bx$.
The lower these utilization levels are, the higher the chance that our system can resist
random fluctuations in resource needs, that is,
the higher the reliability of smooth/uninterrupted operations.

We shall consider an equitable lexmin optimization problem for such systems:
we look for feasible solutions $\bx\in K$ for which the largest component of $r(\bx)$ is as small as possible,
and among those the second largest component of $r(\bx)$ is as small as possible, etc.
To arrive to a precise definition, let us associate to a vector $\bx\in K$ and real $z\in \RR$ the \emph{level set size}
\[
L(z,\bx) ~=~ \left| \{i\in [m]\mid r_i(\bx) ~\geq~ z\}  \right|.
\]
For two vectors $\bx,\by\in K$ we say the $\bx$ \emph{precedes} $\by$ ($\bx\prec \by$) 
if there exists a real threshold $z\in\RR$ such that $L(z,\bx)<L(z,\by)$ and $L(z',\bx)=L(z',\by)$ for all $z'>z$.
Note first that this relation is a pre-order, which is "almost" a complete order:

\begin{proposition}
If for two distinct vectors $\bx,\by\in K$ we have neither $\bx\prec \by$ nor $\by\prec \bx$
then the components of $r(\bx)$ are a permutations of the components of $r(\by)$.
Furthermore, in such a case we have both $\frac12\bx+\frac12\by \prec \bx$ and $\frac12\bx+\frac12\by \prec \by$.
\end{proposition}

\begin{proof}
This is an easy consequence of the definitions. 
\end{proof}

It is also easily seen that it implies the uniqueness of such a lexmin solution:

\begin{corollary}
For every convex set $K \subseteq \RR_\geq^m$ and positive threshold vector $\lambda\in \RR_{> 0}^m$
there exists a unique vector $\bx^*\in K$ such that $\bx^*\prec \by$ for all $\by\in K$, $\by\neq \bx^*$.
\end{corollary}

The algorithmic problem of computing $\bx^*$
(or an approximation of it) was considered in many special cases,
frequently motivated by multi-objective optimization, 
\cite{BR91,KLS92,KO05,Lus98,MO92,Ogr97,Ogr97a,OS02,OS03,OS06,OT03}.
An equivalent formulation was considered by Nace and Orlin in \cite{NO07}
who showed that $\bx^*$ can be computed by solving a series of linear programs,
whenever $K$ is the set of nonnegative feasible solutions 
to a system of linear equalities and inequalities with rational coefficients.

In what follows, we focus on transshipment problems and show that $\bx^*$
can be derived by a much simpler polynomial time algorithm,
based on the characterization of feasibility of such problems 
by  Gale and Hoffman in \cite{Gal57,Hof60}.
This algorithm implies also the uniqueness of the lexmin solution 
independently of the above arguments.
 
For symmetric directed graphs 
(or in other words, for undirected graphs) 
the above problem was solved much earlier, in \cite{GG84}.
The present work can be viewed as an extension of those results 
to arbitrary directed graphs. 

\section{Lexmin Transshipment Problems}
Let  $G = (V, E)$  be a directed graph (digraph)
with weighted vertices (nodes)  $d = (d_v \mid v \in V)$  and
with capacities on the directed edges (arcs) $\lambda = (\lambda_e \mid e \in E)$.
Let us index sets  $V = \{v_1, \ldots, v_n\}$  and  $E = \{e_1, \ldots, e_m\}$
arbitrarily to make $d$  and $\lambda$  vectors.
We interpret  $G$  as a communication or transportation network,
(for example, a pipeline) in which every vertex  $v \in V$
is a node of production if $d_v > 0$,
or consumption if $d_v < 0$, or
just transition  $d_v = 0$  of some material
(in particular, information), while
every edge  $e=(u,v) \in E$  is a one-way
communication line, road, or pipe of capacity  $\lambda_e > 0$.
We call vertex $u$ the \emph{tail} of arc $e=(u,v)$ and $v$ is called its \emph{head}.
We assume that the process is stationary, that is,
the same amount of material
is produced, consumed, and transported in a unit time.
A triplet $(G, d, \lambda)$  defines a {\em transhipment problem}.
We will search a solution of this problem among
nonnegative {\em flow} vectors $x = (x_e \mid e \in E)$, where
$x_e \geq 0$  is interpreted
as the amount of material (information) transported along
arc  $e$  in a unit time.
It will be convenient to introduce the vector of ratios
$r = (r_e = x_e / \lambda_e \mid e \in E)$.

\subsection*{Feasible, weakly feasible, and balanced flows}
We will call flow  $x$ {\em weakly feasible}
(or a {\em weak solution}) to problem $(G, d, \lambda)$
if for each node  $v \in V$  the algebraic sum of flows  $x_e$
for all directed edges  $e$  incident to $v$  equals  $d_v$.
To formulate this more accurately
recall the definition of the {\em node-arc incidence}
$|V| \times |E|$  {\em matrix} $A_G$ of  $G$, in which
the entries  $inc(v,e)$  take values  $\pm 1$  and  $0$
according to the following definition:

\begin{equation}
  \mbox{inc}(v,e)=\left\{
  \begin{aligned}
    +1&, &&\text{if node} \;\; v\ \;\; \text{is the tail of} \;\; e;\\
    -1&, &&\text{if node} \;\; v\ \;\; \text{is the head of} \;\; e;\\
     0&, &&\text{if} \;\; v \;\; \text{and} \;\;  e \text{\;\; are not incident}.
  \end{aligned}\right.
\end{equation}

Thus, a flow  $x$ is weakly feasible for
$(G, d, \lambda)$  if and only if   

\begin{equation}
A_G \; x = d.
\end{equation}

Clearly, problem $(G, d, \lambda)$ may  have a weakly feasible solution
only if  $\sum_{v \in V} d_v = 0$.
We will always assume this equality in the sequel.

\medskip

Furthermore, a flow  $x$  is called {\em feasible}
(or a {\em solution})  to problem $(G, d, \lambda)$
if it is weakly feasible and  $x \leq \lambda$,
that is, $0 \leq  x_e \leq \lambda_e$  for all  $e \in E$.

\medskip

In this paper we study weakly feasible flows
that are not necessarily feasible;
in other words, inequality  $x_e > \lambda_e$,
or equivalently, $r_e > 1$, may hold for some
(maybe, for all) arcs $e \in E$.
Yet, we view large ratios  $r_e$  as unwanted
(potential sources of jams or accidents)
and try to reduce them by minimizing
$r = (r_e = x_e / \lambda_e \mid e \in E)$  lexicographically,
as defined in Section \ref{s0}:
Consider all weakly feasible flows
minimizing  $r^0 = \max (r_e \mid e \in E)$;
among them consider those that minimize the second largest ratio  $r^1$
(equality  $r^0 = r^1$  may hold);  etc.
In  $|E|$  iterations we obtain a unique lexmin vector  $r^*$.
However, this does not yet imply the uniqueness of
the corresponding lexmin flow vector
(which will be proven in Section \ref{s2}; see also Section \ref{s0}).
Meanwhile, we shall assume that there may be several such vectors and
call each of them a {\em balanced} flow, 
or a {\em lexmin solution} to problem  $(G, d, \lambda)$).
By definition, a balanced flow is weakly feasible
and it is feasible whenever a feasible flow exists.

Note that in the latter case the lexicographical minimization is
not that well motivated.
Still, it makes sense, e.g., for a telecommunication network, where
the operator needs to define a routing with respect to a given
traffic demand such that
the network load is fairly distributed among the network links;
see, e.g., \cite{NO07}  for  more details and applications.

\medskip

We will construct a balanced flow
by a simple iterative algorithm
based on the Gale \cite{Gal57} and Hoffman \cite{Hof60}
criterion of solvability of transshipment problems 
and on an arbitrary polynomial time maximum flow algorithm.
Thus,  our algorithm is polynomial too 
whenever the input vectors $d$  and  $\lambda$  are rational.

\subsection*{Cuts}
A cut  $C$  in a digraph  $G = (V,E)$  is defined
as a proper partition of  $V$;
in other words, it is an ordered pair of node-sets
$C = (V', V'')$  such that  $V' \neq  \emptyset$, $V'' \neq  \emptyset$,
$V' \cap V'' = \emptyset$, and $V' \cup V'' = V$.

\smallskip

We say that an arc $e  = (v',v'')$  is in  $C$
if  $v' \in V'$  and  $v'' \in V''$.
Obviously, each cut $C$  is uniquely defined by the set  $E_C$  of its arcs,
but not every subset  $E' \subseteq E$  is associated with a cut.
For each cut  $C$ we introduce the following real numbers
{\em deficiency} $d_C$, {\em capacity} $\lambda_C$, {{\em ratio} $r_C$, and {\em flow} $x_C$,
as follows:

\begin{equation}
\label{e-cut}
\begin{aligned}
d_C &= \sum_{v \in V'} d_v = - \sum_{v \in V''} d_v,\\
\lambda_C &= \sum_{e \in E_C} \lambda_e,\\
r_C &= \frac{d_C}{\lambda_C}, \mbox{ and}\\
x_C &= \sum_{e \in E_C} x_e.
\end{aligned}
\end{equation}

Consider a problem $(G, d, \lambda)$ with positive capacities, and a cut $C$.
Obviously, $\lambda_C = 0$  if and only if  $E_C = \emptyset$.
A cut  $C$  is called  {\em fatal}  if
$d_c > 0$  while  $\lambda_C = 0$;
in this case  $r_C = +\infty$.
A cut $C$ is called {\em deficient} if  $d_C  > \lambda_C$,
or equivalently,  $r_C > 1$.
A cut realizing  $\max (r_C \mid C \in G)$  is called {\em critical}.
Note that such a cut always exists, unlike
a fatal or a deficient cut.
Obviously, a fatal cut  ($r_C = +\infty$)
is deficient ($r_C > 1$).
A critical cut is also deficient
whenever a latter exists.
Note finally that $d_C$  is positive
for all three considered types of cuts.

\medskip

In Section \ref{s1} it will be shown that
problem  $(G, d, \lambda)$  is weakly solvable
(resp., solvable)  if and only if it has
no fatal (resp., deficient) cut.

\subsection*{How to obtain the balanced flow? A plan}
First we will give a polynomial algorithm
constructing a critical cut $C^0$.
Obviously,  $\max (r_e \mid e \in E) = r^0 \geq  r_{C^0}$.
Yet, we hope that equality  $r^0 = r_{C^0}$
will hold for the balanced flow  $x^*$.
Obviously, we must load all edges of cut  $C^0$  uniformly:
$x^*_e = r_{C^0} \lambda_{C^0}$  for all  $e \in E_{C^0}$.
(This is why flow  $x^*$  is  called ``balanced".)
Then,  flow  $x^*_{C^0}$   through  $C^0$
will be equal to  $\sum_{e \in C^0} r_{C^0} \lambda_e = d_{C^0}$,
as required, and  $r_e = r_{C^0}$  for all  $e \in C^0$.
Otherwise, whenever edges of  $C^0$  are loaded not uniformly,
$r_e$  will exceed  $r_{C^0}$  for some $e \in C^0$,
contradicting the assumption that $x^*$  is the balanced flow.

Since  $x^*_e$  are uniquely determined  for  $e  \in E_{C^0}$,
we can delete all these edges from  digraph $G$ and
recompute vector  $d$  accordingly.
Thus, the  original problem  $(G, d, \lambda) = (G^0, d^0, \lambda^0)$
is reduced to  $(G^1, d^1, \lambda^1)$.
Repeat the same procedure for this new problem,
getting a critical cut  $C^1$  and  $r^1 = r_{C^1}$; etc.
In at most  $|E|$  steps flow  $x^*$  will be determined
for all edges of digraph  $G$.

Clearly, inequality  $r^0 = r_{C^0} \geq  r_{C^1} = r^1$
would justify this iterative procedure.
Indeed, it implies monotonicity:  $r^0 \geq r^1 \geq \ldots$,
which in its turn immediately implies that  $x^*$
is the (unique) balanced flow.
However, if we could have  $r_{C^1} > r_{C^0}$  then
it may be not optimal to fix the uniform flow through  $C^0$.

Inequality  $r_{C^0} \geq  r_{C^1}$  could be verified
by a direct computation, as it was done in \cite{GG84}
for the case of symmetric digraphs.
Yet Endre Boros noticed that 
it is simpler to derive this inequality from
the Gale-Hoffman criterion of solvability  \cite{Gal57,Hof60}.
This plan is realized in the next two sections.

\section{Gale-Hoffman's inequalities; a criterion
of solvability for transshipment problems}
\label{s1}
Clearly,  problem  $(G, d, \lambda)$  has no weak solution
if it contains a fatal cut.
Otherwise, if there are no fatal cuts,
in the next section we will construct the balanced flow,
which is weakly feasible by definition.
Thus, the following criterion holds.

\begin{proposition}
Problem  $(G, d, \lambda)$  is weakly solvable
if and only if it contains no fatal cut.
\qed
\end{proposition}

In late  1950-s  Gale \cite{Gal57} and Hoffman \cite{Hof60}
obtained a simple criterion of weak solvability.
It is based on a reduction of the problem
to a two-pole network problem and
applying the {\em max flow -  min cut theorem} to it;
see \cite{FF56} and also \cite{FF62}.

To each transshipment problem  $(G, d, \lambda)$
we assign a two-pole network problem
$(G, d, \lambda; s,t) = (G', \lambda')$  as follows. Define
$V^+ = \{v \in V \mid d_v > 0\}$,
$V^- = \{v \in V \mid d_v < 0\}$,
and  $D = \sum_{v \in V^+} d_v = -\sum_{v \in V^-} d_v$.
Add to  $G$  two new nodes  $s$  and  $t$
and also a directed edge
\begin{itemize}
\item[]
$e_u = (s,u)$
of capacity  $\lambda_{e_u} = d_u$  from  $s$  to each  $u \in V^+$;
\item[]
$e_w = (w,t)$  of capacity  $\lambda_{e_w} = - d_w$
from each  $w \in V^-$  to  $t$.
\end{itemize}
By definition, all these capacities are positive.

The obtained two-pole network $(G, d, \lambda; s,t) = (G', \lambda')$
has two special cuts
$C_s : \{s\} - (V \cup \{t\})$  and
$C_t : (\{s\} \cup V) - \{t\}$
of capacity  $D$   each.
Clearly, the set of all other cuts in  $(G', \lambda')$
is in one-to-one correspondence with  the set of cuts of  $(G, d, \lambda)$.
Furthermore, it is easy to verify that
$$\lambda'_{C'} = D + (\lambda_C - d_C)$$
for every cut  $C$  of  $(G, d, \lambda)$   and  the corresponding
cut  $C'$  of  $(G', \lambda') $.
Thus, the capacity of a minimum cut in
the two-pole network  $(G', \lambda')$
(which is equal to the value
of a maximum flow from  $s$  to $t$  in it) is  
$$c(G,d,\lambda;s,t) = \min (D,  D + (\lambda_C - d_C) \mid C \in G).$$
By definition, it takes value  $D$  if and only if
$\lambda_C \geq d_C$  for every cut  $C$  in  $G$.
Hence, $(G, d, \lambda)$  is solvable if and only if the capacity
$c(G,d,\lambda;s,t)$  of the corresponding two-pole network equals  $D$.
Thus, we obtain the following statement.

\begin{theorem} (\cite{Gal57,Hof60})
\label{t-GH}
Problem $(G, d, \lambda)$  is solvable
if and only if it contains no deficient cut.
\qed
\end{theorem}

In other words, inequalities  $d_C \leq \lambda_C$,
or equivalently $r_C \leq 1$, must  hold for every cut $C$  in  $G$.
These theorem and inequalities are referred to as 
{\em the Gale-Hoffman criterion and inequalities}
for transshipment problems.

Note that  $d_C \leq \lambda_C$  holds automatically 
if  $d_C \leq 0$, since by ddefinition, $0 \leq \lambda_C$  for any cut  $C$.  

\smallskip

To verify solvability of  $(G, d, \lambda)$
we can apply any max-flow algorithm
to the corresponding two-pole network  $(G, d, \lambda; s,t) = (G', \lambda')$   and
compute its capacity  $c = c(G', \lambda')$  by  the algorithm.
if  $c = D$,  the  answer is positive and $(G, d, \lambda)$  is solvable;
otherwise, if  $c < D$, then the answer is negative and $(G, d, \lambda)$  is not solvable.
In the latter case the algorithm outputs
a deficient cut  $C^0$ (such  that  $d_{C^0} > \lambda_{C^0}$).
Moreover, $C^0$  maximizes  $d_C - \lambda_C$  among all cuts of $G$.
It is important to remark, however, that  $C^0$
may not maximize the ratio  $r_{C} = d_C / \lambda_C$.

Note also that the above algorithm is polynomial,
provided all parameters  $d$ and  $\lambda$  are rational.
This follows immediately from
the similar claim for the Ford-Fulkerson maxflow- mincut algorithm.

\section{Solving parametrical problem $(G, d, z \lambda)$
to construct a balanced flow}
\label{s2}
\subsection*{Verifying weak solvability}
Let us parametrize the capacities and consider
problem $(G, d, z \lambda)$, where factor  $z$
is a  positive rational number.

If there exists a fatal cut $C^0$  in  $(G, d, \lambda)$  then
$r_{C^0} = \infty$  and   $C^0$  will remain fatal for any $z$,
implying that $(G, d, z \lambda)$  is not even weakly solvable.

To verify the existence of a fatal cut
we apply the method suggested in the previous section.
Fix a large  $z$,  say $z = M = D / \min (\lambda_e  \mid e \in E)$.
Consider the two-pole network with
capacities  $M \lambda_e$  for all  $e  \in E$.
(Note that capacities
$\lambda_{e_u} = d_u$  for every  $u \in V^+$   and
$\lambda_{e_w} = - d_w$  for each  $w \in V^-$
do not depend on  $z$.)
A fatal flow exists in  $G$  if and only if
the  capacity of the obtained two-pole network
is strictly less than  $D$, which can be  checked by any max-flow algorithm.
From now on, assume that there is no fatal cut in  $(G, d, \lambda)$,
that is, $r_c  < + \infty$  for any cut   $C$  in  $G$.

\subsection*{Minmax ratio and flows}
Given a weakly feasible flow  $x$  in  $(G, d, \lambda)$
maximize the ratio  $r_e  =  x_e / \lambda_e$  for  $e \in E$
and then minimize these maxima over all weakly feasible flows

\begin{equation}
\label{e-r0}
r^0 = \min (\max (r_e) \mid e \in E) \mid
\mbox{ over all weakly feasible flows in} \; (G, d, \lambda)).
\end{equation}

Number  $r^0$   will be called
the {\em minmax ratio} and a weakly feasible flow realizing
minimum in (\ref{e-r0})
will be called a {\em minmax} flow.
In particular, the balanced flow  $x^*$  has this property.

The next claim follows immediately from Gale-Hoffman's Theorem.

\begin{corollary}
\label{c1}
Problem   $(G, d, z \lambda)$  is solvable if and only if  $z \geq r^0$.
\qed
\end{corollary}

To compute  $r^0$  in polynomial time,
consider networks  $(G, d, z \lambda)$  and
the corresponding two-pole networks  $G', z \lambda')$.
(Again,
$\lambda_{e_u} = d_u$  for each  $u \in V^+$   and
$\lambda_{e_w} = - d_w$  for every  $w \in V^-$, independently of  $z$.)
Apply the dichotomy to  $z \in [0; M]$
and a  max-flow algorithm
to determine the minimum  $z$  such that  problem  $(G, d, z \lambda)$  is solvable.
Obviously, this minimum  $z$  is exactly  $r^0$.
This procedure is polynomial provided vectors  $d$  and  $\lambda$
are rational. 

\subsection*{Critical cuts}
Recall definitions given by formulas of  (\ref{e-cut}).
A cut  $C^0$  maximizing the ratio  $r_C$  over all cuts of  $G$
is called {\em critical} for $(G, d, \lambda)$.

The next two statement related to flows, cuts, and ratios of
a problem $(G, d, \lambda)$  will be important.
The first one is an immediate consequence of (\ref{e-r0}) and Theorem \ref{t-GH}.

\begin{corollary}
\label{c2}
The ratio of a critical cut  $C^0$
and the minmax ratio  $r^0$  are equal:
$r_{C^0} = \max(r_C \mid  C$ in $G) = r^0$.
\qed
\end{corollary}

\begin{proposition}
\label{p2}
Every weakly feasible minmax flow  $x$
must load each critical cut  $C^0$  uniformly,  that is,
$x_e = r^0 \lambda_e$  must hold for every  $e \in E_{C^0}$.
\end{proposition}

\proof
If  $x_e = r^0 \lambda_e$  for every  $e \in C^0$  then

$$x_{C^0} = \sum_{e \in E_{C^0}} x_e =
r^0 \sum_{e \in E_{C^0}} \lambda_e = r^0 \lambda_{C^0} = d_{C^0}.$$

\noindent
Equality  $x_C = d_C$  must hold for any weakly feasible flow and cut  $C$,
in particular, for a critical cut  $C^0$.
In the latter case,
if  $x_{e'} < r^0 \lambda_{e'}$  holds for some  $e' \in E_{C^0}$  then
$x_{e''} > r^0 \lambda_{e''}$  will hold for some other  $e'' \in E_{C^0}$,
and thus, $x$  is not a minmax flow.
This contradiction proves the claim.
\qed

\smallskip

To find a critical cut for  $(G, d, \lambda)$,
consider problem  $(G, d, r^0 \lambda)$.
By Corollary \ref{c1}, for this problem we have:
$r_C \leq 1$  for every cut  $C$  and
there is a cut  $C^0$  with   $r_{C^0} = 1$.
There  may exist several such cuts  and
each of them is critical, that is,
its ratio will exceed  $1$
whenever we replace factor  $r^0$  by  $r^0 - \epsilon$,
with an arbitrarily small positive $\epsilon$.

Consider problem  $(G, d, z \lambda)$  with
$z = r^0 - 1/M$ and apply a max-flow algorithm
to the corresponding two-pole problem.
We obtain a cut  $C^0$  that is critical for $(G, d, z \lambda)$
and for  $(G, d, \lambda)$,  too.

\subsection*{Constructing a (unique) balanced flow}
Given, a problem $(G, d, \lambda)$,
find in polynomial time  $r^0$  and a critical cut  $C^0$.
(Note that  $r^0$  is well-defined, but there  may be several critical cuts
whose edge-sets differ.)
By Corollary \ref{c2}, the ratios are equal: $r^0 = r_{C^0}$.
Each weakly feasible  minmax ratio $x$  must load  $C^0$  uniformly,
that is,  $x_e = r^0 \lambda_e$  must hold for every  $e \in C^0$.
Then, equality  $x_{C^0} = d_{C^0}$  holds.

Since we know exactly the values  of  $x_e$  for each $e \in E_{C^0}$,
we can eliminate  $E_{C^0}$  from  $E$  and recompute vector  $d$  accordingly.
Thus, we replace the original problem  $(G, d, \lambda) = (G^0, d^0, \lambda^0)$
by a new one  $(G^1, d^1, \lambda^1).$

\begin{proposition}
\label{p2}
Inequality  $r^1 \leq r^0$  holds.
Moreover, the equality  $r^1 = r^0$  holds whenever
the critical cut  $C^0$  was not unique
and  $E_{C^{0'}} \setminus  E_{C^0} \neq \emptyset$
for some other critical cut  $C^{0'}$  in  $(G, d, \lambda)$.
\end{proposition}

In \cite{GG84} this statement was proven
for symmetric digraphs by arithmetical computations.
In fact, only a sketch of a proof was given,
because the volume of the  paper was strictly limited
by the rules of the journal.
However, such computations are not needed, since both statements
of Proposition \ref{p2} follow from the Gale-Shapley Theorem
and its corollaries given above.

\proof
Suppose that $r^1 > r^0$.
Then for  a  critical cut  $C^1$  in $(G^1, d^1, \lambda^1)$
we have  $r_{C^1} =  r^1 > r^0$.
Hence, for  a weakly feasible minmax flow  $x$
in $(G, d, \lambda)$
there is an edge  $e \in E_{C^1}$  such
$x_e / \lambda_e \geq r_1 > r_0$, which contradicts
the definition of  $r_0$.

Suppose that  $r^1 < r^0$, while
there exists an edge
$e_0 \in E_{C^{0'}} \setminus  E_{C^0}$.
Then, on the first iteration we could choose
the critical cut  $C^{0'}$  rather than  $C^0$
and, as we know, equality $x_{e_0} = r^0 \lambda_{e_0}$
must hold for every weakly feasible minmax ratio  $x$,
in particular, for each  balanced flow.
Yet, this is not possible if   $r^1 < r^0$, which is a contradiction.
\qed

\smallskip

This proposition implies that
in fact we must fix  $x_e = r^0 \lambda_e$
for every edge  $e$  that belongs to a critical cut.
Yet, the result will be the same if we consider
all these cuts together or one by one.
For example, after eliminating  $E_{C^0}$,
we can repeat the procedure with the same factor  $z = r^0 - 1/M$
searching for another critical cut  $C^{0'}$
with the same ratio  $r_{C^{0'}} = r_{C^0} = r_0$.
In there is one, we fix
$x^*_e = r^0 \lambda_e$  for all $e \in E_{C^{0'}}$  and repeat.
If there is no one  then  $r*1 < r^0$
and we search for  $r^1$  applying the dichotomy to the reduced problem.
In particular, these arguments imply uniqueness of the balanced flow.

\begin{corollary}
Each  problem  $(G, d, \lambda)$  has a unique lexmin solution.
\qed
\end{corollary}

By definition, the balanced flow  $x^*$  remains
a weakly feasible minmax flow on every iteration.
Thus, we can repeat  the  above procedure and
in at most  $|E|$  iterations
we will determine a flow  $x^*_e$  on every edge  $e  \in E$.
If  vectors  $d$ and $\lambda$  are  rational then  
every iteration is known to be polynomial
and, hence, the whole procedure is polynomial too.

\section{Related works}
\subsection*{Symmetric digraphs}
A digraph  $G = (V, E)$  with capacities $\lambda$  is called {\em symmetric}
if to  each edge  $e' = (u,w) \in E$  of capacity  $\lambda_{e'}$,
a unique oppositely directed edge
$e'' = (w,u) \in E$  of the same capacity, $\lambda_{e''} = \lambda_{e'}$, is assigned.
A symmetric digraph can be replaced by a non-directed graph.
In this case the balanced flow was defined and constructed
by an iterative polynomial algorithm in \cite{GG84}.
Note, however, that in that  paper
the key inequality  $r^0 \geq r^1$  was not derived
from the Gale-Hoffman criterion.
Instead it was proved, by a  pretty long algebraic computation,
which was only sketched in the paper, 
since its size was limited in accordance with the rules of the journal.

\subsection*{Balanced flows in networks with monomial conductance}
In \cite{Gur10,Gur20,GG84,GG87,GG92,OGG86,OGG86a}
networks with ``boundary conditions" $A_G \; x = d^0$ 
and monomial conductance  $x_e = \lambda_e y_e^\tau$  were considered.
Here   $y_e$  is the voltage or pressure drop on
a directed edge   $e  \in E$, while
$\tau$  is a positive real parameter common for all  $e \in E$.
For example, $\tau = 1$   and  $\tau = 1/2$
correspond respectively to the classic Ohm law and
square conductance law, which is typical in hydraulics and gas dynamics.
Let us remark that in \cite{Gur10,GG84,GG87}
only the case of symmetric digraphs
(that is, non-directed graphs) and
in \cite{Gur10,GG87} only the two-pole networks
(that is, the first Kirchhoff law, $d_v = 0$,
was required for all notes except two) were considered.
In \cite{GG84,GG87} it was shown
that the flow tends to the balanced one,
as  $\tau \rightarrow 0$; see \cite{Gur10} for  more details
and \cite{Gur20} for generalization to the directed case.

\section*{Acknowledgements}
The author is thankful to Endre Boros who suggested
to derive monotonicity of the minmax ratios,
$r_{C^0} \geq r_{C^1} \geq \ldots$,
from the Gale-Hoffman inequalities, rather that
to verify it by direct algebraic transformations,
as it was  done in \cite{GG84} in case of symmetric digraphs. 
He also suggested an alternative approach to uniqueness given in Section \ref{s0}.
\newline
This research was done
within the framework of the HSE University Basic Research Program
and supported by the RSF grant  20-11-20203.



\begin{thebibliography}{99}


\bibitem{BG87}
D. Bertsekas and  R. Gallager  Data Networks (1987) Prentice-Hall, Englewood Cliffs.

\bibitem{BR91}
 R.E. Burkard  and F. Rendl,
 Lexicographic bottleneck problems, Operations Research Letters10 (1991) 303--308.

%

\bibitem{FF56}
 L.R. Ford and D.R. Fullkerson, Maximal Flow Through a Network,
 Canadian J. Math. 8 (1956) 399--404.

\bibitem{FF62}
L.R. Ford and D.R. Fullkerson,
Flows in networcs, Princeton Univ. Press, 1962.

\bibitem{Gal57}
D. Gale, A theorem of flows in networks,
Pacific J. of Math 7 (1957) 1073--1082.

\bibitem{Gur10}
V. Gurvich,
Metric and ultrametric spaces of resistances,
Discrete Applied Mathematics 158 (2010) 1496--1505.

\bibitem{Gur20}
V. Gurvich,
Metric and ultrametric inequalities for resistances in directed graphs,
Preprint  http://arxiv.org/abs/2009.14316 (2020) 21 pp.

\bibitem{GG84}
A.D. Gvishiani and V.A. Gurvich,
Balanced flow in multipole networks,
Soviet Phys. Dokl. 29:4 (1984) 268--270.

\bibitem{GG87}
A.D. Gvishiani and V.A. Gurvich,
Metric and ultrametric spaces of resistances,
Russian Math. Surveys  42:2 (1987) 235--236.


\bibitem{GG92}
A.D. Gvishiani and V.A. Gurvich,
Dynamic Problems of Pattern Recognition
and Applied Convex Programming in Applications, 355 pp.
(in Russian) Moscow, ``Nauka" (Science) Publishers, 1992.


\bibitem{Hof60}
A.J. Hoffman,
Some recent applications of the theory of linear inequalities
to extremal combinatorial analysis,
Proc. Symp. in Appl. Math., Amer. Math. Soc. (1960) 113--127.

\bibitem{KLS92}
R.S. Klein, H. Luss, and D.R. Smith,
A lexicographic minimax algorithm for multiperiod resource allocation,
Mathematical Programming, 55 (1992)  213--234.

\bibitem{KO05}
A. Krzemienowski and W. Ogryczak,
On Extending the LP Computable Risk Measures to Account Downside Risk,
Comput. Optim. Appl. 32:1-2 (2005) 133--160.

\bibitem{Lus98}
H. Luss,  On equitable resource allocation problems:
A lexicographic minimax approach, Operations Research 47 (1999) 361--378.

\bibitem{MO92}
E. Marchi and J.A. Oviedo,
Lexicographic optimality in the multiple objective linear programming:
the nucleolar solution,
Eur. J. Opnl. Res.57 (1992) 355--359.

\bibitem{NO07}
D. Nace and J.B. Orlin,
Lexicographically Minimum and Maximum Load Linear Programming Problems,
Operations Research 55:1 (2007) 182--187.

\bibitem{OGG86}
I.F. Obraztsov, A.D. Gvishiani, and V.A. Gurvich,
On the theory of monotone schemes,
Soviet Math. Dokl. 34:1 (1986) 162--166.

\bibitem{OGG86a}
I.F. Obraztsov, A.D. Gvishiani, and V.A Gurvich,
Calculation of networks and
dual problems of convex programming,
Soviet Math. Dokl. 33:2 (1986) 491--496.


\bibitem{Ogr97}
W. Ogryczak,
Linear and Discrete Optimization with Multiple Criteria:
Preference Models and Applications to Decision Support
(in Polish) Warsaw University Press, Warsaw, 1997.

\bibitem{Ogr97a}
W. Ogryczak, On the lexicographic minimax approach to location problems,
Eur. J. Opnl. Res. 100 (1997) 566--585.



\bibitem{OS02}
W. Ogryczak and T. Sliwinski,
On equitable approaches to resource allocation problems:
the conditional minimax solution,
J. Telecom. Info. Tech.:3/02 (2002) 40--48.

\bibitem{OS03}
W. Ogryczak and Sliwinski,
On solving linear programs with the ordered weighted averaging objective,
Eur. J. Opnl. Res. 148 (2003) 80--91.

\bibitem{OS06}
W. Ogryczak and T. Sliwinski,
On Direct Methods for Lexicographic Min-Max Optimization,
ICCSA 3 (2006)  802--811.


\bibitem{OT03}
W. Ogryczak and A. Tamir,
Minimizing the sum of the $k$ largest functions in linear time,
Information Processing Letters 85 (2003) 117--122.




\end{thebibliography}
\end{document}